\documentclass[twoside,centertags]{amsart}

\usepackage[cmtip,color,all]{xy}
\usepackage{amsmath}
\usepackage{amsfonts}
\usepackage{amssymb}
\usepackage{amsthm}
\usepackage{graphicx}
\usepackage{mathrsfs}

\newtheorem{theorem}{Theorem}[section]
\newtheorem{corollary}[theorem]{Corollary}
\newtheorem{lemma}[theorem]{Lemma}
\newtheorem{proposition}[theorem]{Proposition}
\theoremstyle{definition}
\newtheorem{definition}[theorem]{Definition}
\theoremstyle{remark}

\newcommand{\OF}{\widetilde{F}}
\newcommand{\F}{F}
\newcommand{\Z}{\mathbb{Z}}
\newcommand{\R}{\mathbb{R}}

\begin{document}

\title{A remark on configuration spaces of two points} 
\author{George Raptis}
\address{\newline
G. Raptis \newline
Universit\"{a}t Regensburg, 
Fakult\"{a}t f\"{u}r Mathematik \newline
93040 Regensburg, Germany}
\email{georgios.raptis@ur.de}
\author{Paolo Salvatore}
\address{\newline
P. Salvatore \newline
Dipartimento di Matematica, 
Universit\`{a} di Roma Tor Vergata \newline
Via della Ricerca Scientifica, 
00133 Roma, Italy}
\email{salvator@mat.uniroma2.it}
\begin{abstract}
We prove a homotopy invariance result for a certain covering space of the space of ordered configurations 
of two points in $M \times X$ where $M$ is a closed smooth manifold and $X$ is any fixed aspherical space 
which is not a point. 
\end{abstract}

\maketitle

\section{Introduction}

It is known that the homotopy type of the ordered configuration space $F_2(M)$ of two distinct points in a closed manifold $M$ is not determined by the homotopy type of $M$. Longoni 
and the second-named author found a counterexample to this homotopy invariance problem in \cite{LS}. The counterexample is given by the pair of homotopy equivalent 3-dimensional 
lens spaces $L_{7,1}$ and $L_{7,2}$. In this case, it turns out that the universal covering spaces of $F_2(L_{7,1})$ and $F_2(L_{7,2})$ 
are also not homotopy equivalent. More pairs of lens spaces have been examined by Evans-Lee \cite{Lee}, providing evidence for the conjecture that any pair of non-homeomorphic 
lens spaces gives a counterexample. On the other hand, there is a positive result by Levitt \cite{Le} who proved the homotopy invariance of $F_2(M)$ when $M$ is 2-connected.

The nature of the counterexample suggested the modified question of the homotopy invariance of $F_2(M)$ with respect to the {\em simple} homotopy type of $M$. This question remains open. An easy way of producing 
simple-homotopy equivalent manifolds is by taking product with $S^1$: the product property of the Whitehead torsion shows that a homotopy equivalence $f: M \stackrel{\simeq}{\to} N$ 
yields a simple-homotopy equivalence  $f \times \mathrm{id}: M \times S^1 \stackrel{\simeq_s}{\to} N \times S^1$. In this paper, we consider the space $\F_2^{S^1}(M \times S^1)$ of pairs of points in $M \times \R$ which lie in distinct $\Z$-orbits. This 
defines a $(\Z \times \Z)$-covering space over $F_2(M \times S^1)$. A special case of our main result is that the homotopy type of this space is a homotopy invariant of $M$.

More generally, if $X$ is a fixed aspherical space which is \emph{not} the one-point space, then the homotopy type of a certain covering space of $\F_2(M \times X)$ is 
homotopy invariant in $M$ (Theorem \ref{main}). If $X$ is also contractible, this implies the homotopy invariance of $F_2(M \times X)$ (Corollary \ref{cor}). These 
statements are false, of course, when $F_2(X) = \varnothing$, i.e., when $X$ consists of a single point. The proof of Theorem \ref{main} uses a description of the covering space 
of $\F_2(M \times X)$ as a homotopy pushout (Proposition \ref{htpy}) and the fiber homotopy invariance of the spherical tangent bundle of a closed smooth manifold \cite{STM}.

\section{Configurations of two points in a product of spaces}

Throughout this section, $M$ is a closed smooth manifold and $X$ is a path-connected Hausdorff 
space with a basepoint $x \in X$ and a universal covering $p: \widetilde{X} \to X$.

\subsection{Preliminaries.} The configuration space $\F_2(M) \subset M \times M$ consists of ordered pairs of distinct points in $M$, i.e., 
$$\F_2(M) = \{(m_1, m_2) \in M^2 \ | \ \ m_1 \neq m_2 \}.$$

\noindent Assume that $M$ has a Riemannian metric $d$. For $\epsilon > 0$, we consider the following open subspaces of $M \times M$, 
$$F_2(M)_{\epsilon} : \ = \{(m_1, m_2) \in F_2(M) \ | \  d(m_1, m_2) < \epsilon \}$$ 
and
$$DT(M)_{\epsilon} : \ = \{(m_1, m_2) \in M \times M \ | \  d(m_1, m_2) < \epsilon\}.$$
There is a (homotopy) pushout square
\begin{equation} \label{decomp_MxM}
\xymatrix{
F_2(M)_{\epsilon} \ar[r] \ar[d] & \F_2(M) \ar[d] \\
DT(M)_{\epsilon} \ar[r] & M \times M
}
\end{equation}
For $\epsilon$ small enough, the projection $F_2(M)_{\epsilon} \to M$, $(m_1, m_2) \mapsto m_1$, is homotopy equivalent, fiberwise over $M$, to the spherical tangent 
bundle of $M$ (see also \cite{Le}). The fiber homotopy type of the spherical tangent bundle of $M$ depends only on the homotopy type of $M$ by results of \cite{STM, Du}.  
On the other hand, for $\epsilon$ small, the corresponding projection $DT(M)_{\epsilon} \to M$ is a homotopy equivalence and fiber homotopy equivalent to the disk tangent bundle of $M$. 

\subsection{Orbit 2-configurations in a product.}

Let $G: = \pi_1(X, x) \times \pi_1(X, x)$. We consider the following covering space of the configuration space of two points in $M \times X$.

\begin{definition}
The \textit{$X$-orbit configuration space} $F_2^X(M \times X)$ of two points in $M \times X$ is the covering space of $F_2(M \times X)$ defined by 
$$\F_2^X(M \times X) : \ = \{\big( (m_1, x_1), (m_2, x_2) \big) \in \F_2(M \times \widetilde{X}) \ \arrowvert \ (m_1, p(x_1)) \neq (m_2, p(x_2)) \}.$$
\end{definition}

The space $\F_2^X(M \times X)$ admits a natural free action by the group $G$ and the quotient is the configuration space $\F_2(M \times X)$. 
For $M = \ast$, the space $F_2^X(M \times X)$ is the standard orbit configuration space of $\widetilde{X}$, denoted $\widetilde{\F}_2(X)$. There is a pushout 
square 
\begin{equation} \label{key_pushout}
\xymatrix{
\F_2(M) \times \OF_2(X) \ar[r] \ar[d] & \F_2(M) \times \widetilde{ X }^2 \ar[d] \\
M^2 \times \widetilde{\F}_2(X) \ar[r] & \F_2^X(M \times X)
}
\end{equation}
where the maps are the obvious open inclusions. These maps respect the respective actions of $G$ and there is an induced pushout square
\[
\xymatrix{
\F_2(M) \times \F_2(X) \ar[r] \ar[d] & \F_2(M) \times X^2 \ar[d] \\
M^2 \times \F_2(X) \ar[r] & \F_2(M \times X)
}
\]
Combining the pushout decompositions in \eqref{decomp_MxM} and \eqref{key_pushout}, we obtain the following

\begin{lemma} \label{decomp}
There is a $G$-equivariant homeomorphism
\newline
$$ (DT(M)_{\epsilon} \times \OF_2(X)) \bigcup_{F_2(M)_{\epsilon} \times \OF_2(X)} (\F_2(M) \times \widetilde{X}^2) \stackrel{\cong}{\longrightarrow} \F^X_2(M \times X).$$
\end{lemma}
\begin{proof}
This follows easily from the diagram of $G$-equivariant maps
\[
\xymatrix{
F_2(M)_{\epsilon} \times \OF_2(X) \ar[r] \ar[d] & \F_2(M) \times \OF_2(X) \ar[r] \ar[d] & \F_2(M) \times \widetilde{X}^2 \ar[d] \\
DT(M)_{\epsilon} \times \OF_2(X) \ar[r] & M^2 \times \OF_2(X) \ar[r] & \F_2^X(M \times X)
}
\]
Since both squares are pushouts, so is also the composite square. 
\end{proof}

\begin{corollary} \label{eq-htpy}
Assume that $F_2(X)$ is non-empty (i.e., $X$ has at least two points). Then there is a pushout of $G$-equivariant maps
\[
\xymatrix{
(F_2(M)_{\epsilon} \times \OF_2(X)) \bigcup_{F_2(M)_{\epsilon} \times G} (DT(M)_{\epsilon} \times G)
\ar[r] \ar[d] & DT(M)_{\epsilon} \times \OF_2(X) \ar[d] \\
 (\F_2(M) \times \widetilde{X}^2) \bigcup_{F_2(M)_{\epsilon} \times G} (DT(M)_{\epsilon} \times G)  \ar[r] 
 & \F_2^X(M \times X)
}
\] 
\end{corollary}
\begin{proof}
Let $q: G \to \OF_2(X)$ be the inclusion of an orbit where $G$ is regarded as a discrete topological group. Consider the following diagram:
\begin{small}
\[
\xymatrix{
F_2(M)_{\epsilon} \times G \ar[d]^{\mathrm{id} \times q} \ar[r] & DT(M)_{\epsilon} \times  G \ar[d] \ar[dr]^{\mathrm{id} \times q} & \\
F_2(M)_{\epsilon} \times \OF_2(X) \ar[r] \ar[d] & (F_2(M)_{\epsilon} \times \OF_2(X)) \cup_{F_2(M)_{\epsilon} \times G} (DT(M)_{\epsilon} \times G) \ar[d] \ar@{-->}[r] & DT(M)_{\epsilon} 
\times \OF_2(X) \ar[d] \\
\F_2(M) \times \widetilde{X}^2 \ar[r] & (\F_2(M) \times \widetilde{X}^2)\cup_{F_2(M)_{\epsilon} \times G} (DT(M)_{\epsilon} \times G)  \ar@{-->}[r] & \F_2^X(M \times X)
}\]\end{small}

\noindent Note that all of the maps respect the corresponding $G$-actions. The squares on the left are pushouts by definition. The composite bottom square 
is a pushout by Lemma \ref{decomp}. Therefore the bottom right square is also a pushout, as required.
\end{proof}

\subsection{Homotopy Invariance.}

The somewhat complicated diagram in Corollary \ref{eq-htpy} can be simplified at the expense of losing $G$-equivariance. First, let $\epsilon > 0$ be small enough so that
the closed inclusion of the subspace of $F_2(M)_{\epsilon}$ which consists of those pairs of points which are exactly $(\epsilon/2)$-apart,
$$ST(M) : \ = \{(m_1, m_2) \in M \times M \ |  \ d(m_1, m_2) = \epsilon/2 \} \stackrel{j}{\hookrightarrow} F_2(M)_{\epsilon},$$  
is a homotopy equivalence, the projection $ST(M) \to M$, $(m_1, m_2) \mapsto m_1$, is fiber homotopy equivalent to 
the spherical tangent bundle of $M$, and the projection $DT(M)_{\epsilon} \to M$ is a homotopy equivalence. We denote $DT(M) : = \ DT(M)_{\epsilon}$.

We obtain the following homotopy pushout decomposition of $\F_2^X(M \times X)$. Here \textit{homotopy pushout} is always considered with respect to the weak homotopy 
equivalences. 

\begin{proposition} \label{htpy}
Suppose that $\widetilde{X}$ is weakly contractible and $F_2(X) \neq \varnothing$ (i.e, $X$ has at least two points). Let $q: \ast \to \OF_2(X)$ be the inclusion of a point. Then the 
space $\F_2^X(M \times X)$ is weakly equivalent to the homotopy pushout of the maps
\[
 M^2 \longleftarrow (ST(M) \times \OF_2(X)) \bigcup_{ST(M) \times \{\ast\}} DT(M) \times \{\ast\} \longrightarrow DT(M) \times \OF_2(X)
\]
which are defined by the projection away from $\OF_2(X)$ and the point $q$.
\end{proposition}
\begin{proof}
The proof is similar to that of Corollary \ref{eq-htpy}.  Consider the following commutative diagram 
\[
\xymatrix{
ST(M) \times  \ast \ar[d]^{j \times q} \ar[r] & DT(M) \times  \ast \ar[d] \ar[dr]^{\mathrm{id} \times q} & \\
F_2(M)_{\epsilon} \times \OF_2(X) \ar[r] \ar[d] & F_2(M)_{\epsilon} \times \OF_2(X) \cup_{ST(M)} \ DT(M)
\ar@{-->}[r] \ar[d] & DT(M) \times \OF_2(X) \ar[d] \\
\F_2(M) \times \widetilde{X}^2\ar[r] & \F_2(M) \times \widetilde{X}^2 \cup_{ST(M)} \ DT(M) \ar@{-->}[r] & \F_2^X(M \times X)
}
\]
The two squares on the left are pushouts by definition. The top map is a cofibration, therefore they are also homotopy pushouts 
(see, e.g., \cite[Appendix, Proposition 4.8]{BV}). The bottom composite square is a pushout by Lemma \ref{decomp}. This pushout decomposition 
of $F_2^X(M \times X)$ arises from an open covering defined by two open subsets and therefore it defines a homotopy pushout (see also \cite{DI}
for more general results). It follows that the bottom right square is also a homotopy pushout. 

There is an obvious commutative diagram 
\[
\xymatrix{
DT(M) \times \OF_2(X)  \ar@{=}[r] & DT(M) \times \OF_2(X) \\
(\F_2(M)_{\epsilon} \times \OF_2(X)) \cup_{ST(M)} DT(M) \ar@{=}[r] \ar[d] \ar[u] & (F_2(M)_{\epsilon} \times \OF_2(X)) \cup_{ST(M)} DT(M) \ar[d] \ar[u] \\
\F_2(M) \times \widetilde{X}^2 \cup_{ST(M)} DT(M) \ar[r]^{\sim} & M^2 
 }
\]
where the bottom map is a weak homotopy equivalence using that $\widetilde{X} \to \ast$ is a weak homotopy equivalence, $ST(M) \stackrel{j}{\simeq} \F_2(M)_{\epsilon}$ is a homotopy equivalence, 
and the homotopy pushout in Diagram \eqref{decomp_MxM}. Therefore the homotopy pushouts of the vertical pairs of maps are weakly homotopy equivalent. Similarly, 
they can be identified with the required homotopy pushout using the homotopy equivalence $ST(M) \stackrel{j}{\simeq} \F_2(M)_{\epsilon}$ and the fact that $ST(M) \subset
DT(M)$ is a cofibration. 
\end{proof}

\begin{theorem} \label{main}
Suppose that $X$ has a weakly contractible universal covering space and $\F_2(X) \neq \varnothing$ (i.e., $X$ has at least two points). If $M$ and $N$ are homotopy equivalent 
closed smooth manifolds, then the spaces $\F_2^X(M \times X)$ and $\F_2^X(N \times X)$ are weakly homotopy equivalent.
\end{theorem}
\begin{proof}
By the fiber homotopy invariance of the spherical tangent bundle \cite{STM, Du}, there is a homotopy commutative square 
\[
 \xymatrix{
 ST(M) \ar[r]^{\simeq} \ar[d] & ST(N) \ar[d] \\
 M \ar[r]^{\simeq} & N
 }
\]
where the vertical maps are (any of) the (two homotopic) projections and the horizontal maps are homotopy equivalences. The projection $DT(M) \to M$ is a homotopy equivalence and,  under this identification, the inclusion $DT(M) \subset M \times M$ is homotopic to the diagonal inclusion $\Delta: M \to M \times M$. Thus, the homotopy pushout decomposition in Proposition \ref{htpy} is weakly homotopy invariant in $M$ and the result follows.
\end{proof}

\begin{corollary} 
Let $X$, $M$ and $N$ be as in Theorem \ref{main}. Suppose that $\pi_1(X)$ is finite. Then there is 
a zig-zag of maps connecting $F_2(M \times X)$ and $F_2(N \times X)$ and inducing isomorphisms in 
rational homology.
\end{corollary}
\begin{proof}
The claim is obvious when $M$ and $N$ are $0$-dimensional. If the dimension is positive, the zig-zag of maps is as follows
$$\F_2(M \times X) \leftarrow \F_2^X(M \times X) \simeq_w \F_2^X(N \times X) \rightarrow \F_2(N \times X)$$
where in the middle is the weak homotopy equivalence from Theorem \ref{main} and the other two maps are the natural projections. 
These two maps are finite covering maps and it is easy to check that they induce bijections on $\pi_0$. Therefore they induce isomorphisms
between the rational homology groups. 
\end{proof}

\begin{corollary} \label{cor}
Let $M$ and $N$ be homotopy equivalent closed smooth manifolds.
\begin{itemize} 
 \item[(a)] Suppose that $X$ is weakly contractible and $\F_2(X) \neq \varnothing$. Then $\F_2(M \times X)$ and $\F_2(N \times X)$
 are weakly homotopy equivalent.
 \item[(b)] $\F_2^{S^1}(M \times S^1)$ and $\F_2^{S^1}(N \times S^1)$ are homotopy equivalent.
\end{itemize}
\end{corollary}

\begin{corollary}
The spaces $\F_2^{S^1}(L_{7,1} \times S^1)$ and $\F_2^{S^1}(L_{7,2} \times S^1)$ are homotopy equivalent. 
\end{corollary}

Since $L_{7,1}$ and $L_{7,2}$ are not homeomorphic, the spaces $L_{7,1} \times S^1$ and $L_{7,2} \times S^1$ are also not homeomorphic by results of 
\cite{KR} (see, e.g., the proof in \cite[p. 177]{KR}). However, they are simple-homotopy equivalent because the Whitehead torsion of $f \times \mathrm{id}_{S^1}$ 
vanishes for every homotopy equivalence $f$. In \cite{LS}, it was shown that the orbit configuration spaces $\widetilde{\F}_2(L_{7,1})$ and $\widetilde{\F}_2(L_{7,2})$ are not homotopy 
equivalent, thus disproving the homotopy invariance of configuration spaces. It remains open whether the configuration spaces $F_2(L_{7,1} \times S^1)$ and $F_2(L_{7,2} \times S^1)$ 
are homotopy equivalent and whether, more generally, the correspondence $M \mapsto F_2(M \times S^1)$ is homotopy invariant. Based on the properties of the 
Whitehead torsion, this problem relates to the general question about the homotopy invariance of configuration spaces with respect to simple-homotopy equivalences.

\end{document}